\documentclass{amsart}
\usepackage{amsmath,amssymb,amsthm}
\usepackage{url}
\usepackage{graphicx}
\usepackage{verbatim}

\newtheorem{theorem}{Theorem}

\newtheorem{corollary}[theorem]{Corollary}

\newtheorem{example}[theorem]{Example}

\newtheorem{proposition}[theorem]{Proposition}

\graphicspath{    {converted_graphics/}    {/}}

\begin{document}
\title{Selections and their Absolutely Continuous Invariant Measures}
\thanks{The research of the authors was supported by NSERC grants. }
\date{\today }

\author[A. Boyarsky]{Abraham Boyarsky }
\address[A. Boyarsky]{Department of Mathematics and Statistics, Concordia University,
1455 de Maisonneuve Blvd. West, Montreal, Quebec H3G 1M8, Canada}
\email[A. Boyarsky]{boyar@alcor.concordia.ca}

\author[P. G\'ora]{Pawe\l\ G\'ora }
\address[P. G\'ora]{Department of Mathematics and Statistics, Concordia University,
1455 de Maisonneuve Blvd. West, Montreal, Quebec H3G 1M8, Canada}
\email[P. G\'ora]{pawel.gora@concordia.ca}

\author[Zh. Li]{Zhenyang Li }
\address[Zh. Li]{Department of Mathematics and Statistics, Concordia University,
1455 de Maisonneuve Blvd. West, Montreal, Quebec H3G 1M8, Canada}
\email[Zh. Li]{zhenyangemail@gmail.com}

\begin{abstract}
Let $I=[0,1]$ and consider disjoint closed regions $G_{1},....,G_{n}$ in $%
I\times I$ and subintervals $I_{1},......,I_{n},$ such that $G_{i}$ projects
onto $I_{i.}$ We define the lower and upper maps $\tau _{1},$ $\tau _{2}$ by
the lower and upper boundaries of $G_{i},i=1,....,n,$ respectively. We
assume $\tau _{1}$, $\tau _{2}$ to be piecewise monotonic and preserving
continuous invariant measures $\mu _{1}$ and $\mu _{2}$, respectively. Let $%
F^{(1)}$ and $F^{(2)}$ be the distribution functions of $\mu _{1}$ and $\mu
_{2}.$ The main results shows that for any convex combination $F$ of $%
F^{(1)} $ and $F^{(2)}$ we can find a map $\eta $ with values between the
graphs of $\tau _{1}$ and $\tau _{2}$ (that is, a selection) such that $F$
is the $\eta $-invariant distribution function. Examples are presented. We
also study the relationship of the dynamics of multi-valued maps to random
maps.
\end{abstract}

\maketitle

 {Department of Mathematics and Statistics, Concordia University,
1455 de Maisonneuve Blvd. West, Montreal, Quebec H3G 1M8, Canada}

 \smallskip

E-mails: {boyar@alcor.concordia.ca}, {pawel.gora@concordia.ca},
{zhenyangemail@gmail.com}.

\section{INTRODUCTION}

A multivalued discrete time dynamical system is specified by a map $\Gamma $
whose image at any point in $I$ is a measurable subset of $I=[0,1].$ The
graph of $\Gamma $ is the set: $G=$ $\{(x,y)\in I$ $X$ $I$ $\left\vert y\in
\Gamma (x)\}.\right. $ Such maps have application in economics \cite{Aum},
modeling, and rigorous numerics \cite{Kac} and in dynamical systems \cite%
{Art2000,Art2004}. The objective of this note is to study maps whose graphs
are inside $G$ and which possess absolutely continuous invariant measures
(acim).

We assume $G$ consists of disjoint closed regions $G_{1},....,G_{n}$ in $I$ $%
X$ $I$ where $G_{i}$ projects onto the subinterval $I_{j}$. We define $\tau
_{1},$ $\tau _{2}$ to be the lower and upper boundaries of $G$ and assume $%
\tau _{1},$ $\tau _{2}$ to be in the class of piecewise expanding, piecewise
$C^{2}$ maps from $I$ into $I.$ Thus, $\tau _{1},\tau _{2}$ have acims with
probability density functions (pdf), $f_{1}$ and $f_{2}.$

Motivating examples are presented in section 2. The first constructs a
selection with desired properties in the case where the upper and lower maps
are piecewise linear. The second example shows that if the class of
transformations is restricted only to the graphs of the lower and upper maps
$\tau _{1}$ and $\tau _{2}$, that is, the the set $G$ consists only of the
graphs of these two maps, then there is no transformation that has pdf equal
to a convex combination of $f_{1}$ and $f_{2}.$

In Section 4 we present the main result: assuming that the lower edge and
upper edge maps $\tau _{1}$ and $\tau _{2}$ are piecewise monotonic and that
their invariant distribution functions $F^{(1)}$ and $F^{(2)}$ are
continuous and for any $0<\lambda <1$ the convex combination $F=\lambda
F^{(1)}+(1-\lambda )F^{(2)}$ is a homeomorphism of the unit interval, then
there exists a piecewise monotonic selection $\eta $, $\tau _{1}\leq \eta
\leq \tau _{2}$, preserving the distribution function $F$.

In Section 5 we present an approach to finding selections based on
conjugation : if $\tau _{1}$ is piecewise linear and the $\tau _{2}$ is
conjugated to $\tau _{1}$ then, for any convex combination $f$ of $f_{1}$
and $f_{2}$ we can find a map $\tau $ with values between the graphs of $%
\tau _{1}$ and $\tau _{2}$ such that $f$ is the pdf associated with $\tau .$%
In fact $\tau $ is also a conjugacy of $\tau _{1\text{ }}.$ A similar result
can be proved for more general maps which are shaped like the tent map. In
Section 6 we study the relationship between the dynamics of multi-valued maps
and random maps. In particular, we consider a multi-valued map consisting of two graphs,
 and show that in general the statistical long term behaviour of an arbitrary
map between the graphs of the multi-valued map cannot be achieved by a
position dependent random map based on the maps defining the multi-valued map.
A number of positive examples are also presented.

\section{Motivating example}

Let us consider a multi-valued map $T$ with lower edge selection $\tau_1$
and upper edge selection $\tau_2$ as in Figure \ref{fig:convex_counter}. If $%
\tau_1$ preserves a density $f_1$ and $\tau_2$ preserves a density $f_2$,
then we ask whether for any convex combination $f=\alpha\cdot
f_1+(1-\alpha)\cdot f_2$, $0<\alpha<1$, we can find a selection of $T$ which
preserves the density $f$. We present a counter example showing that if $%
T=\{\tau_1,\tau_2\}$ ($T$ is two-valued), then it maybe impossible.

\begin{figure}[h]
% float placement: (h)ere, page (t)op, page (b)ottom, other (p)age
\centering
% file name: E:/0Docs/0TeX/PapersT/Abrahams_MultiValued/convex_counter.eps
\includegraphics[bb=20 118 575
673,width=3.92in,height=3.92in,keepaspectratio]{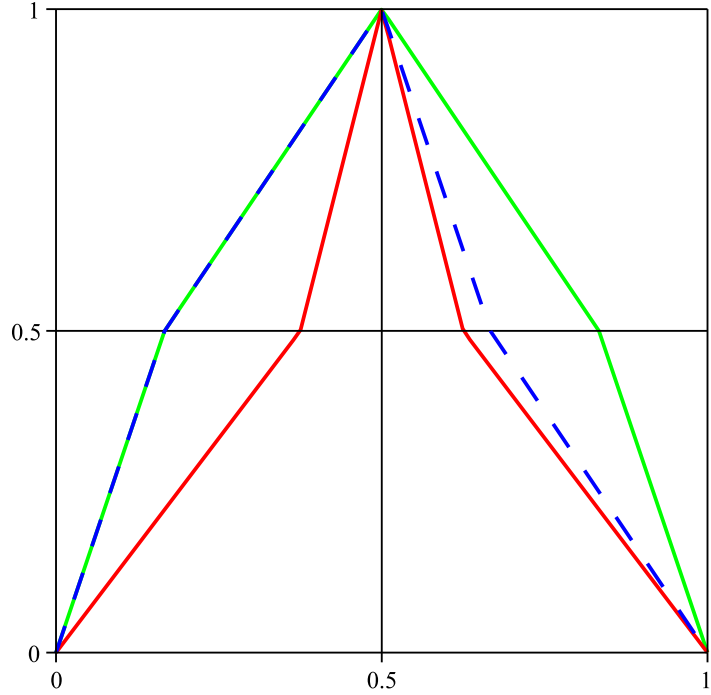}
\caption{Two valued map of Example \protect\ref{counterEX}}
\label{fig:convex_counter}
\end{figure}

\begin{example}
\label{counterEX}
\end{example}

Let
\begin{equation*}
\tau_1(x)=%
\begin{cases}
\frac 4 3 x , & \ \ 0\le x<\frac 3 8 \ ; \\
4 x-1 , & \ \ \frac 38\le x<\frac 12 \ ; \\
-4 x +3 , & \ \ \frac 12 \le x<\frac 58 \ ; \\
-\frac 43 x +\frac 43 , & \ \ \frac 58 \le x\le 1 \ ,%
\end{cases}%
\end{equation*}
and
\begin{equation*}
\tau_2(x)=%
\begin{cases}
3 x , & \ \ 0\le x<\frac 16 \ ; \\
\frac 32 x+\frac 14 , & \ \ \frac 16\le x<\frac 12 \ ; \\
-\frac 32 x +\frac 74 , & \ \ \frac 12 \le x<\frac 56 \ ; \\
-3 x +3 , & \ \ \frac 56 \le x\le 1 \ .%
\end{cases}%
\end{equation*}
The invariant densities are $f_1= \frac 32\chi_{[0,1/2]}+\frac
12\chi_{[1/2,1]}$ and $f_2= \frac 23\chi_{[0,1/2]}+\frac 43\chi_{[1/2,1]}$,
correspondingly. Thus, the Lebesgue measure density is a convex combination
of $f_1$ and $f_2$, $1=\frac 25\cdot f_1+ \frac 35\cdot f_2$. For a two
branch map $\tau$ to leave Lebesgue measure invariant it is necessary to
satisfy
\begin{equation*}
|\tau^{\prime -1}_1(x)|+ |\tau^{\prime -1}_2(x)|=1 \ ,
\end{equation*}
at the preimages $\tau^{-1}_1(x)$ and $\tau^{-1}_2(x)$ of every point $x$,
which is impossible.

\textbf{Remark:} If we allow at least one of the branches of the map $\tau$
to be between the maps $\tau_1$ and $\tau_2$, then we can achieve the
invariance of Lebesgue measure. For example map given below preserves
Lebesgue measure. Its graph is shown in Figure \ref{fig:convex_counter}
using dashed lines.
\begin{equation*}
\tau(x)=%
\begin{cases}
\tau_2(x) , & \ \ 0\le x<\frac 12 \ ; \\
-3 x +\frac 52 , & \ \ \frac 12 \le x<\frac 46 \ ; \\
-\frac 32 x +\frac 32 , & \ \ \frac 46 \le x\le 1 \ .%
\end{cases}%
\end{equation*}
%%%%%%%%%%%%%%%%%%%%%%%%%%%%%%%%%%%%%%%%%%%%%%%%%%%%%%%%%%%%%%%%%%%%%%%%%%%%%%%%%%

\section{Another motivating example for "tent" like maps}

We assume that both maps $\tau_1$, $\tau_2$ are increasing on $[0,1/2]$ and
decreasing on $[1/2,1]$ and have value 0 at 0 and 1 and value 1 at $1/2$. We
do not assume here that lower map $\tau_1$ is conjugated to the upper map $%
\tau_2$. Let us assume that $\tau_1$ preserves measure $\mu_1$ and $\tau_2$
preserves measure $\mu_2$, not necessarily absolutely continuous. Let $%
F^{(1)},F^{(2)}$ be the distribution functions of measures $\mu_1$, $\mu_2$,
respectively ($F^i(x)=\mu_i([0,x])$, $i=1,2$). Let $\mu=\lambda\mu_1+(1-%
\lambda)\mu_2$, $0<\lambda<1$, and let $F$ be the distribution function of $%
\mu$:
\begin{equation*}
F(x)=\mu([0,x])\ .
\end{equation*}
We are looking for map $\eta$ satisfying $\tau_1\le\eta\le\tau_2$ preserving
distribution function $F$ (or equivalently measure $\mu$).

We introduce the function $s:[0,1/2]\to[1/2,1]$ which relates the branches $%
\eta_1,\eta_2$ of map $\eta$. Let
\begin{equation}  \label{eta2}
\eta_2(x)=\eta_1(s^{-1}(x))\ .
\end{equation}
Frobenius-Perron equation gives
\begin{equation*}
F(x)=F(\eta^{-1}_1(x))+1-F(\eta^{-1}_2(x))\ ,
\end{equation*}
or
\begin{equation*}
F(\eta_1(z))=F(z)+1-F(s(z))\ ,
\end{equation*}
or
\begin{equation*}
F(s(z))=1+F(z)-F(\eta_1(z))\ ,
\end{equation*}
which allows us to find $s$ once $\eta_1$ is given
\begin{equation}  \label{map_s}
s(z)=F^{-1}(1+F(z)-F(\eta_1(z)))\ .
\end{equation}
Thus, once we construct $\eta_1$ satisfying $\tau_{1,1}\le\eta_1\le%
\tau_{2,1} $ we obtain $\eta_2$ and have to check if it satisfies required
inequalities. We will show that $\eta_1$ can be chosen in such a way that $%
\eta_2$ is between $\tau_1$ and $\tau_2$.

Assumptions:
\begin{equation}  \label{ine}
\tau_{2,1}^{-1}\le\eta^{-1}_1\le\tau_{1,1}^{-1}\ ,
\end{equation}
(equivalent to $\tau_{1,1}\le\eta_1\le\tau_{2,1}$),

\begin{equation}  \label{convex}
F=\lambda F^{(1)}+(1-\lambda) F^{(2)}\ .
\end{equation}

Frobenius-Perron equation gives
\begin{equation*}
F^{(1)}(x)=F^{(1)}(\tau_{1,1}^{-1}(x))+1-F^{(1)}(\tau_{1,2}^{-1}(x))\ ,
\end{equation*}
\begin{equation*}
F^{(2)}(x)=F^{(2)}(\tau_{2,1}^{-1}(x))+1-F^{(2)}(\tau_{2,2}^{-1}(x))\ ,
\end{equation*}
\begin{equation*}
F(x)=F(\eta^{-1}_1(x))+1-F(\eta^{-1}_2(x))\ ,
\end{equation*}
or
\begin{equation}  \label{no2}
\begin{split}
F^{(1)}(\tau_{1,2}^{-1}(x))&=F^{(1)}(\tau_{1,1}^{-1}(x))+1-F^{(1)}(x)\ , \\
F^{(2)}(\tau_{2,2}^{-1}(x))&=F^{(2)}(\tau_{2,1}^{-1}(x))+1-F^{(2)}(x)\ , \\
F(\eta^{-1}_2(x))&=F(\eta^{-1}_1(x))+1-F(x)\ .
\end{split}%
\end{equation}
We want to show:
\begin{equation*}
\tau_{1,2}^{-1}(x)\le \eta^{-1}_2(x)\le \tau_{2,2}^{-1}(x) \ ,
\end{equation*}
or equivalently
\begin{equation}  \label{toshow}
F(\tau_{1,2}^{-1}(x))\le F(\eta^{-1}_2(x))\le F(\tau_{2,2}^{-1}(x)) \ .
\end{equation}

First, we will show that it is possible to choose $\eta_1$ in such a way
that
\begin{equation}  \label{firstpart}
F(\tau_{1,2}^{-1}(x))\le F(\eta^{-1}_2(x))\ .
\end{equation}
Using (\ref{convex}) and (\ref{no2}) we obtain an inequality equivalent to (%
\ref{firstpart})
\begin{equation}  \label{hej1}
\lambda F^{(1)}(\tau_{1,2}^{-1}(x))+(1-\lambda)
F^{(2)}(\tau_{1,2}^{-1}(x))\le F(\eta^{-1}_1(x))+1-F(x)\ ,
\end{equation}
and another also equivalent to (\ref{firstpart})
\begin{equation}  \label{hej2}
\begin{split}
\lambda& \left[F^{(1)}(\tau_{1,1}^{-1}(x))+1- F^{(1)}(x)\right]+(1-\lambda)
F^{(2)}(\tau_{1,2}^{-1}(x)) \\
&\le \lambda
F^{(1)}(\eta^{-1}_1(x))+(1-\lambda)F^{(2)}(\eta^{-1}_1(x))+1-\lambda
F^{(1)}(x)-(1-\lambda) F^{(2)}(x)\ ,
\end{split}%
\end{equation}
or
\begin{equation}  \label{hej3}
\begin{split}
\lambda& F^{(1)}(\tau_{1,1}^{-1}(x))+(1-\lambda) F^{(2)}(\tau_{1,2}^{-1}(x))
\\
&\le \lambda
F^{(1)}(\eta^{-1}_1(x))+(1-\lambda)F^{(2)}(\eta^{-1}_1(x))-(1-\lambda)[
F^{(2)}(x)-1]\ ,
\end{split}%
\end{equation}
or, using again (\ref{no2})
\begin{equation}  \label{hej3}
\begin{split}
\lambda& F^{(1)}(\tau_{1,1}^{-1}(x))+(1-\lambda) F^{(2)}(\tau_{1,2}^{-1}(x))
\\
&\le \lambda
F^{(1)}(\eta^{-1}_1(x))+(1-\lambda)F^{(2)}(\eta^{-1}_1(x))-(1-\lambda)\left[%
F^{(2)}(\tau_{2,1}^{-1}(x)) -F^{(2)}(\tau_{2,2}^{-1}(x))\right]\ ,
\end{split}%
\end{equation}
or, still equivalent to (\ref{firstpart})
\begin{equation}  \label{hej4}
\begin{split}
\lambda& \left[F^{(1)}(\tau_{1,1}^{-1}(x))-F^{(1)}(\eta^{-1}_1(x))\right]%
+(1-\lambda)\left[F^{(2)}(\tau_{2,1}^{-1}(x))-F^{(2)}(\eta^{-1}_1(x))\right]
\\
&\le (1-\lambda)\left[ F^{(2)}(\tau_{2,2}^{-1}(x))-F^{(2)}(%
\tau_{1,2}^{-1}(x))\right]\ .
\end{split}%
\end{equation}
Since $\tau_{2,2}^{-1}\ge \tau_{1,2}^{-1}$ the right hand side is positive
independently of a choice of $\eta_1$. Since $\tau_{2,1}^{-1}\le
\eta_1^{-1}\le \tau_{1,1}^{-1}$, the first term on the left hand side is
positive (and zero for $\eta_1=\tau_{1,1}$) and the second term is negative.
This shows that there is an interval $I_1(x)$ touching $\tau_{1,1}(x)$ such
that if we choose $\eta_1(x)$ in this interval, then the inequality (\ref%
{hej4}) and thus (\ref{firstpart}) will be satisfied.

Now, we will show the second part of (\ref{toshow}), i.e., that it is
possible to choose $\eta_1$ in such a way that
\begin{equation}  \label{secpart}
F(\eta^{-1}_2(x))\le F(\tau_{2,2}^{-1}(x))\ .
\end{equation}
As above we change (\ref{secpart}) equivalently to eventually obtain
\begin{equation}  \label{hej5}
\begin{split}
(1-\lambda)&\left[F^{(2)}(\eta^{-1}_1(x))-F^{(2)}(\tau_{2,1}^{-1}(x))\right]
+\lambda \left[F^{(1)}(\eta^{-1}_1(x))-F^{(1)}(\tau_{1,1}^{-1}(x))\right]+ \\
&\le \lambda\left[ F^{(1)}(\tau_{2,2}^{-1}(x))-F^{(1)}(\tau_{1,2}^{-1}(x))%
\right]\ .
\end{split}%
\end{equation}
Since $\tau_{2,2}^{-1}\ge \tau_{1,2}^{-1}$ the right hand side is positive
independently of a choice of $\eta_1$. Since $\tau_{2,1}^{-1}\le
\eta_1^{-1}\le \tau_{1,1}^{-1}$, the first term on the left hand side is
positive (and zero for $\eta_1=\tau_{2,1}$) and the second term is negative.
This shows that there is an interval $I_2(x)$ touching $\tau_{2,1}(x)$ such
that if we choose $\eta_1(x)$ in this interval, then the inequality (\ref%
{hej5}) and thus (\ref{secpart}) will be satisfied.

Now, we will show that there exists an $\eta_1$ satisfying inequalities (\ref%
{ine}), (\ref{hej4}) and (\ref{hej5}).

Note that left hand size of (\ref{hej4})\ $\ = \ -$ left hand side of (\ref%
{hej5}), while both right hand sides are positive. We choose $\eta_1$ such
that the left hand side of (\ref{hej4}) ( and also of (\ref{hej5})) is 0.
Then, both inequalities are satisfied. This is possible, at least if we
assume that $F^{(1)}$, $F^{(2)}$ are continuous.

Here is how we solve for $\eta_1$. We have
\begin{equation*}
\lambda F^{(1)}(\eta^{-1}_1(x))+(1-\lambda)
F^{(2)}(\eta^{-1}_1(x))=(1-\lambda) F^{(2)}(\tau_{2,1}^{-1}(x))+\lambda
F^{(1)}(\tau_{1,1}^{-1}(x))\ ,
\end{equation*}
or
\begin{equation*}
F(\eta^{-1}_1(x))=(1-\lambda) F^{(2)}(\tau_{2,1}^{-1}(x))+\lambda
F^{(1)}(\tau_{1,1}^{-1}(x))\ ,
\end{equation*}
\begin{equation}  \label{eta1inv}
\eta^{-1}_1(x)=F^{-1}\left((1-\lambda) F^{(2)}(\tau_{2,1}^{-1}(x))+\lambda
F^{(1)}(\tau_{1,1}^{-1}(x))\right)\ ,
\end{equation}
assuming $F$ is strictly increasing.

Now, we show that $\eta_1(x)$ defined by equation (\ref{eta1inv}) satisfies
assumptions \ref{ine}, i.e., $\eta_1$ is located between $\tau_{1,1}$ and $%
\tau_{2,1}$. First,
\begin{equation*}
\tau_{2,1}^{-1}\le\eta^{-1}_1
\end{equation*}
is equivalent to
\begin{equation*}
\tau_{2,1}^{-1}\le F^{-1}\left((1-\lambda)
F^{(2)}(\tau_{2,1}^{-1}(x))+\lambda F^{(1)}(\tau_{1,1}^{-1}(x))\right)\ ,
\end{equation*}
or
\begin{equation*}
F(\tau_{2,1}^{-1})\le (1-\lambda) F^{(2)}(\tau_{2,1}^{-1}(x))+\lambda
F^{(1)}(\tau_{1,1}^{-1}(x))\ ,
\end{equation*}
or
\begin{equation*}
\lambda F^{(1)}(\tau_{2,1}^{-1})+ (1-\lambda) F^{(2)}(\tau_{2,1}^{-1})\le
(1-\lambda) F^{(2)}(\tau_{2,1}^{-1}(x))+\lambda F^{(1)}(\tau_{1,1}^{-1}(x))\
,
\end{equation*}
or
\begin{equation*}
F^{(1)}(\tau_{2,1}^{-1})\le F^{(1)}(\tau_{1,1}^{-1}(x))\ ,
\end{equation*}
which is true since $\tau_{2,1} \geq \tau_{1,1}$ and both are increasing. On
the other hand,
\begin{equation*}
\eta^{-1}_1\le\tau_{1,1}^{-1}
\end{equation*}
is equivalent to
\begin{equation*}
F^{-1}\left((1-\lambda) F^{(2)}(\tau_{2,1}^{-1}(x))+\lambda
F^{(1)}(\tau_{1,1}^{-1}(x))\right)\le\tau_{1,1}^{-1}\ ,
\end{equation*}
or
\begin{equation*}
(1-\lambda) F^{(2)}(\tau_{2,1}^{-1}(x))+\lambda
F^{(1)}(\tau_{1,1}^{-1}(x))\le F(\tau_{1,1}^{-1})\ ,
\end{equation*}
or
\begin{equation*}
(1-\lambda) F^{(2)}(\tau_{2,1}^{-1}(x))+\lambda
F^{(1)}(\tau_{1,1}^{-1}(x))\le \lambda F^{(1)}(\tau_{1,1}^{-1}) +
(1-\lambda) F^{(2)}(\tau_{1,1}^{-1})\ ,
\end{equation*}
or
\begin{equation*}
F^{(2)}(\tau_{2,1}^{-1}(x)) \le F^{(2)}(\tau_{1,1}^{-1})\ ,
\end{equation*}
which is true by the same reason as in the first case.

Note that $\eta_1(x)$ defined by equation (\ref{eta1inv}) is increasing and
continuous since all the functions defining $\eta_1(x)$ there are increasing
and continuous. Also,
\begin{equation*}
\eta^{-1}_1(0)=F^{-1}\left((1-\lambda) F^{(2)}(\tau_{2,1}^{-1}(0))+ \lambda
F^{(1)}(\tau_{1,1}^{-1}(0))\right)=F^{-1}(0)=0\ ,
\end{equation*}
\begin{equation*}
\eta^{-1}_1(1)=F^{-1}\left((1-\lambda) F^{(2)}(\tau_{2,1}^{-1}(1))+ \lambda
F^{(1)}(\tau_{1,1}^{-1}(1))\right)=F^{-1}\left(F(1/2)\right)=1/2\ .
\end{equation*}
Actually, for the tent like map we are considering, it is not necessarily
assumed that the maximum is achieved at $1/2$, it can be any point in $(0,1)
$.

\section{Main result}

Our main result is the following theorem.

\begin{theorem}
Let $\Gamma$ be a set valued map from the unit interval into itself. Assume
that the lower edge and upper edge maps $\tau_1$ and $\tau_2$ are piecewise
monotonic, their invariant distribution functions $F^{(1)}$ and $F^{(2)}$
are continuous and for any $0<\lambda<1$ the convex combination $F=\lambda
F^{(1)}+(1-\lambda) F^{(2)}$ is a homeomorphism of the unit interval. Then,
there exists a piecewise monotonic selection $\eta$, $\tau_1\le\eta\le\tau_2$%
, preserving distribution function $F$.
\end{theorem}

\begin{proof}
Without loss of generality, we assume the partition points of $\tau_1$ and $%
\tau_2$ are common $a_0=0<a_1<a_2<\cdots<a_{m}=1$. Let $I_j=[a_{j-1},a_j]$, $%
j=1,2,3,\ldots,m$. On each interval $I_j$, $\tau_{1,j}:=\tau_1|_{_{I_j}}$
and $\tau_{2,j}:=\tau_2|_{_{I_j}}$ share the same monotonicity, where we
understand $\tau_{1,j}$ and $\tau_{1,j}$ as the natural extensions of pieces
of $\tau_1$ and $\tau_2$, respectively. Moreover, we expect our selections
also have the above properties.

For any interval $[a,b]\subseteq [0,1]$, given a monotone continuous
function $h:[a,b]\rightarrow [0,1]$, we define its extended inverse as
follows. We introduce some notations. Let
\begin{equation*}
h^{\max}=\max\left\{h(x)|x\in[a,b]\right\}\, ,
\end{equation*}
and
\begin{equation*}
h^{\min}=\min\left\{h(x)|x\in[a,b]\right\}\, .
\end{equation*}
If $h$ is increasing, then its extended inverse is defined as
\begin{equation*}
\overline{h^{-1}}(x)=%
\begin{cases}
a \ , \text{\ for\ } \ x\in[0,h^{\min}]\ ; \\
h^{-1}(x) \ , \text{\ for\ } \ x\in[h^{\min},h^{\max}]\ ; \\
b \ , \text{\ for\ } \ x\in[h^{\max},1]\ .%
\end{cases}%
\end{equation*}
If $h$ is decreasing, then its extended inverse is defined as
\begin{equation*}
\overline{h^{-1}}(x)=%
\begin{cases}
b \ , \text{\ for\ } \ x\in[0,h^{\min}]\ ; \\
h^{-1}(x) \ , \text{\ for\ } \ x\in[h^{\min},h^{\max}]\ ; \\
a \ , \text{\ for\ } \ x\in[h^{\max},1]\ . \\
\end{cases}%
\end{equation*}

We define the extended inverse of each branch of $\eta$ by
\begin{equation}  \label{mulineq8}
\overline{\eta_{j}^{-1}}(x)=F^{-1}\left(\lambda F^{(1)}(\overline{%
\tau_{1,j}^{-1}}(x))+(1-\lambda) F^{(2)}(\overline{\tau_{2,j}^{-1}}%
(x))\right)\, ,
\end{equation}
where $j=1,2,3,\ldots,m$. $\eta$ defined in this way after the vertical
segments removed has the same number of branches as $\tau_1$ and $\tau_2$,
and each branch of it also has the same monotonicity.

First, we show that $\eta$ is located in between $\tau_1$ and $\tau_2$. For
some $j\in \{1,2,\ldots,m\}$, we show this for the case that $\tau_{1,j}$
and $\tau_{2,j}$ are increasing, the proof for the case that $\tau_{1,j}$
and $\tau_{2,j}$ are decreasing is similar. We need to show the following:
\begin{equation*}
\overline{\tau_{2,j}^{-1}}(x)\leq \overline{\eta_{j}^{-1}}(x)\leq \overline{%
\tau_{1,j}^{-1}}(x)\, ,
\end{equation*}
which is equivalent to
\begin{equation*}
F\left(\overline{\tau_{2,j}^{-1}}(x)\right)\leq F\left(\overline{%
\eta_{j}^{-1}}(x)\right)\leq F\left(\overline{\tau_{1,j}^{-1}}(x)\right)\, ,
\end{equation*}
or using (\ref{mulineq8})
\begin{eqnarray*}
\lambda F^{(1)}(\overline{\tau_{2,j}^{-1}}(x))&+&(1-\lambda)F^{(2)}(%
\overline{\tau_{2,j}^{-1}}(x)) \\
&&\leq\lambda F^{(1)}\overline{(\tau_{1,j}^{-1}}(x))+ (1-\lambda)F^{(2)}(%
\overline{\tau_{2,j}^{-1}}(x)) \\
&&\leq\lambda F^{(1)}(\overline{\tau_{1,j}^{-1}}(x))+ (1-\lambda)F^{(2)}(%
\overline{\tau_{1,j}^{-1}}(x))\, ,
\end{eqnarray*}
which is true since $\overline{\tau_{2,j}^{-1}}(x)\leq \overline{%
\tau_{1,j}^{-1}}(x)$ since $\tau_{1,j}$ and $\tau_{2,j}$ are increasing.

\begin{figure}[h]
% float placement: (h)ere, page (t)op, page (b)ottom, other
\centering
% file name: E:/0Docs/0TeX/PapersT/Abrahams_MultiValued/distribution_functions_new.eps
\includegraphics[bb=20 118 575
673,width=3.92in,height=3.92in,keepaspectratio]{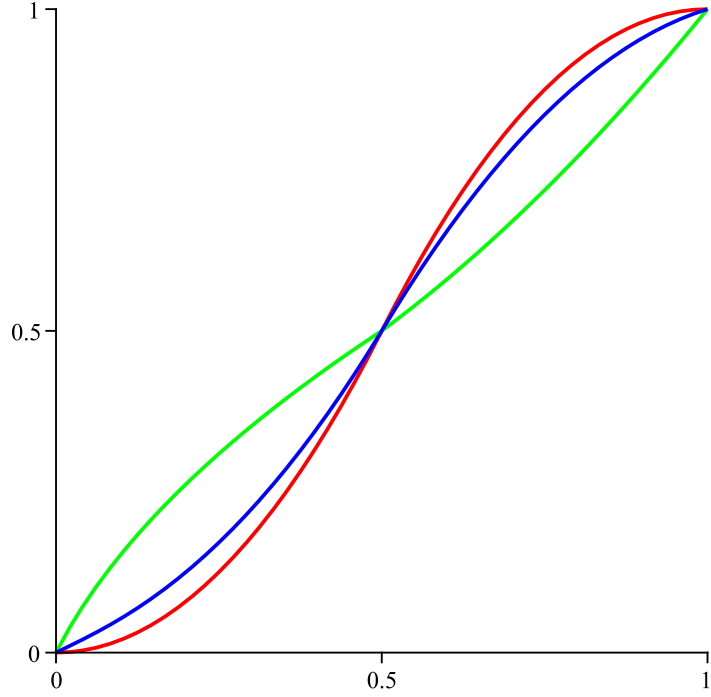}
\caption{Invariant distribution functions $F^{(1)}$, $F^{(2)}$ and $F$.}
\label{fig:distribution_functions_new}
\end{figure}

Second, for any $x\in[0,1]$, using the previous notations, we have
\begin{equation*}
\overline{\eta^{-1}}([0,x])=\bigcup_{j=1}^m[{\overline{\eta_{j}^{-1}}}^l(x),{%
\overline{\eta_{j}^{-1}}}^r(x)]\, ,
\end{equation*}
\begin{equation*}
\overline{\tau_1^{-1}}([0,x])=\bigcup_{j=1}^m[{\overline{\tau_{1,j}^{-1}}}%
^l(x),{\overline{\tau_{1,j}^{-1}}}^r(x)]\, ,
\end{equation*}
\begin{equation*}
\overline{\tau_2^{-1}}([0,x])=\bigcup_{j=1}^m[{\overline{\tau_{2,j}^{-1}}}%
^l(x),{\overline{\tau_{2,j}^{-1}}}^r(x)]\, ,
\end{equation*}
where the bars over inverses of maps $\eta^{-1}$, $\tau_1^{-1}$ and $%
\tau_2^{-1}$ imply that the extended inverses are used for each branch. Note
that all the three maps have the same monotonicity for each corresponding
branch. Moreover, for some $x$, the interval notations appearing on the
right hand side of above preimages may only contain one point. For example,
if $x\in[0,\tau_{1,j}^{\min}]$ where $j\in \{1,2,\ldots,m\}$, then ${%
\overline{\tau_{1,j}^{-1}}}^l(x)={\overline{\tau_{1,j}^{-1}}}^r(x)=a_{j-1}$
when $\tau_{1,j}$ is increasing.

\begin{figure}[h]
% float placement: (h)ere, page (t)op, page (b)ottom, other (p)age
\centering
% file name: E:/0Docs/0TeX/PapersT/Abrahams_MultiValued/eta_map_new_asym.eps
\includegraphics[bb=20 118 575
673,width=3.92in,height=3.92in,keepaspectratio]{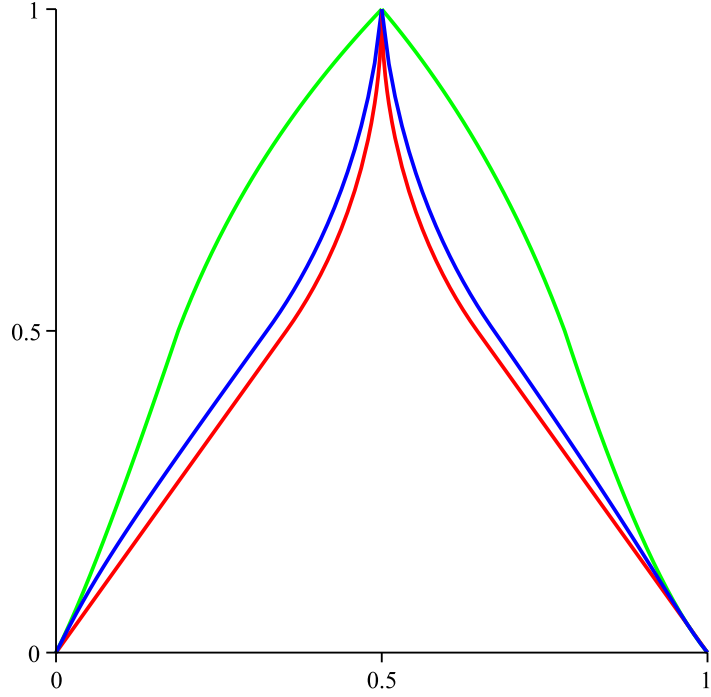}
\caption{Transformations $\protect\tau_1$, $\protect\tau_2$ and $\protect%
\eta $.}
\label{fig:eta_map_new_asym}
\end{figure}

For maps $\tau_{1}$ and $\tau_{2}$, Frobenius-Perron equation gives
\begin{equation*}
F^i(x)=\sum_{j=1}^m \left[F^i\left({\overline{\tau_{i,j}^{-1}}}^r(x)\right)-
F^i\left({\overline{\tau_{i,j}^{-1}}}^l(x)\right)\right]\, ,
\end{equation*}
$i=1,2.$

Now, using (\ref{mulineq8}) and the fact that $\eta$, $\tau_1$ and $\tau_2$
have the same monotonicity on each interval $I_j$, $j\in \{1,2,\ldots,m\}$,
we have
\begin{eqnarray*}
F\left({\overline{\eta_{j}^{-1}}}^r(x)\right)-F\left({\overline{\eta_{j}^{-1}%
}}^l(x)\right) &=&\lambda F^{(1)}\left({\overline{\tau_{1,j}^{-1}}}%
^r(x)\right)+(1-\lambda)F^{(2)}\left({\overline{\tau_{2,j}^{-1}}}^r(x)\right)
\\
&&-\left[\lambda F^{(1)}\left({\overline{\tau_{1,j}^{-1}}}%
^l(x)\right)+(1-\lambda)F^{(2)}\left({\overline{\tau_{2,j}^{-1}}}^l(x)\right)%
\right] \\
&=&\lambda \left[F^{(1)}\left({\overline{\tau_{1,j}^{-1}}}^r(x)\right)-
F^{(1)}\left({\overline{\tau_{1,j}^{-1}}}^l(x)\right)\right] \\
&&+ (1-\lambda)\left[F^{(2)}\left({\overline{\tau_{1,j}^{-1}}}^r(x)\right)-
F^{(2)}\left({\overline{\tau_{1,j}^{-1}}}^l(x)\right)\right]\, .
\end{eqnarray*}
Thus, denoting the measure corresponding to $F$ by $\mu$, we have
\begin{eqnarray*}
\mu\left(\eta^{-1}([0,x])\right)&=&\sum_{j=1}^m F\left({\overline{%
\eta_{j}^{-1}}}^r(x)\right)-F\left({\overline{\eta_{j}^{-1}}}^l(x)\right) \\
&=&\lambda \sum_{j=1}^m\left[F^{(1)}\left({\overline{\tau_{1,j}^{-1}}}%
^r(x)\right)- F^{(1)}\left({\overline{\tau_{1,j}^{-1}}}^l(x)\right)\right] \\
&&+ (1-\lambda)\sum_{j=1}^m\left[F^{(2)}\left({\overline{\tau_{1,j}^{-1}}}%
^r(x)\right)- F^{(2)}\left({\overline{\tau_{1,j}^{-1}}}^l(x)\right)\right] \\
&=&\lambda F^{(1)}(x)+(1-\lambda) F^{(2)}(x) \\
&=& F(x)\, ,
\end{eqnarray*}
which implies indeed the map $\eta$ defined in (\ref{mulineq8}) preserves $F$%
. This completes the proof.
\end{proof}

%%%%%%%%%%%%%%%%%%%%%%%%%%%%%%%%%%%%%%%%%%%%%%%%%%%%%%%%%%%%%%%%%%%%%%%%%%%%%%%%%%%%%%%%%%%%%%%%%%%%%%%%%%%%

Let $\varphi_1$ and $\varphi_2$ be homeomorphisms of $[0,1]$ onto itself
defined as follows:
\begin{equation*}
\varphi_1(x)=%
\begin{cases}
2x^2, & \ \text{for}\ 0\le x< 1/2\ ; \\
1-2(1-x)^2, & \ \text{for}\ 1/2\le x\le 1 \ ,%
\end{cases}%
\end{equation*}

\begin{equation*}
\varphi_2(x)=%
\begin{cases}
-\frac 14+\frac 14\sqrt{1+16x}, & \ \text{for}\ 0\le x< 1/2\ ; \\
\frac 12\left(x^2+\frac 12(x+1)\right), & \ \text{for}\ 1/2\le x\le 1 \ .%
\end{cases}%
\end{equation*}

Then, we define maps $\tau_1=\varphi_1^{-1}\circ T\circ \varphi_1$ and $%
\tau_2=\varphi_2^{-1}\circ T\circ \varphi_2$, where $T$ is the tent map. The
graphs of $\tau_1$ and $\tau_2$ are shown in Figure \ref%
{fig:eta_map_new_asym}. The invariant distribution function for $\tau_1$ is $%
F^{(1)}=\varphi_1$, and the invariant distribution function for $\tau_2$ is $%
F^{(2)}=\varphi_2$ (Corollary 3). Let $\alpha=3/4$ and $F=\alpha
F^{(1)}+(1-\alpha)F^{(2)}$. The distribution functions are shown in Figure %
\ref{fig:distribution_functions_new}. In Figure \ref{fig:eta_map_new_asym}
we show map $\eta$ constructed using formulas (\ref{mulineq8}).

%%%%%%%%%%%%%%%%%%%%%%%%%%%%%%%%%%%%%%%%%%%%%%%%%%%%%%%%%%%%%%%%%%%%%%%%%%%%%%%%%%%%%%%%%%%%%%%%%%%%%%%%%%%%%%%%

\section{ Another method of partially solving the problem}

In this section we generally assume that lower and upper edge maps $\tau_1$
and $\tau_2$ are conjugated and use this conjugation to construct an "in
between" selection.

\begin{proposition}
Let $\tau_1$ and $\tau_2$ be interval $[0,1]$ maps preserving densities $f_1$
and $f_2$, correspondingly, and conjugated by a diffeomorphism (or a at
least absolutely continuous homeomorphism) $h$:
\begin{equation*}
\tau_2=h^{-1}\circ \tau_1\circ h\ .
\end{equation*}
Then,
\begin{equation*}
f_2=(f_1\circ h) \cdot|h^{\prime }|\ .
\end{equation*}
\end{proposition}

\begin{proof}
Generally known.
\end{proof}

\begin{corollary}
If $\tau_1=T$, the tent map, then
\begin{equation*}
f_2=|h^{\prime }|\ ,
\end{equation*}
or equivalently
\begin{equation*}
h(x)= \pm \int_0^x f_2(t) dt\ .
\end{equation*}
\end{corollary}

\begin{proposition}
\label{Pr:conj} Let $\tau_1$ be a piecewise linear Markov map of interval $%
[0,1]$ onto itself preserving density $f_1$. This means that there are
partition $\mathcal{P}$ such that $\mathcal{P}=\{I_i\}_{i=1}^n$ and $%
\tau_1(I_j)$ is a union of consecutive elements of $\mathcal{P}$ for any $%
1\le j\le n$. Then, $f_1$ is piecewise constant $f_1=\sum_{i=1}^n
c_i\chi_{I_i}$. Let $\tau_2$ be a map conjugated to $\tau_1$ by a
diffeomorphism (or a at least absolutely continuous homeomorphism) $h$
preserving partition $\mathcal{P}$,
\begin{equation*}
\tau_2=h^{-1}\circ \tau_1\circ h\ .
\end{equation*}
Then, $\tau_2$ preserves the density
\begin{equation*}
f_2= |h^{\prime }|\cdot\sum_{i=1}^n c_i\chi_{I_i}\ ,
\end{equation*}
and
\begin{equation*}
|h^{\prime }|=f_2\cdot \sum_{i=1}^n \frac 1{c_i}\chi_{I_i}\ .
\end{equation*}
\end{proposition}

\begin{proof}
We have
\begin{equation*}
f_2=(f_1\circ h) \cdot|h^{\prime }|=\sum_{i=1}^n c_i\chi_{I_i}\circ h
\cdot|h^{\prime }|=|h^{\prime }|\cdot\sum_{i=1}^n c_i\chi_{I_i}\ .
\end{equation*}
\end{proof}

\begin{figure}[h]
% float placement: (h)ere, page (t)op, page (b)ottom, other (p)age
\centering
% file name: E:/0Docs/0TeX/PapersT/Abrahams_MultiValued/Markov_example.eps
\includegraphics[bb=20 118 575
673,width=3.92in,height=3.92in,keepaspectratio]{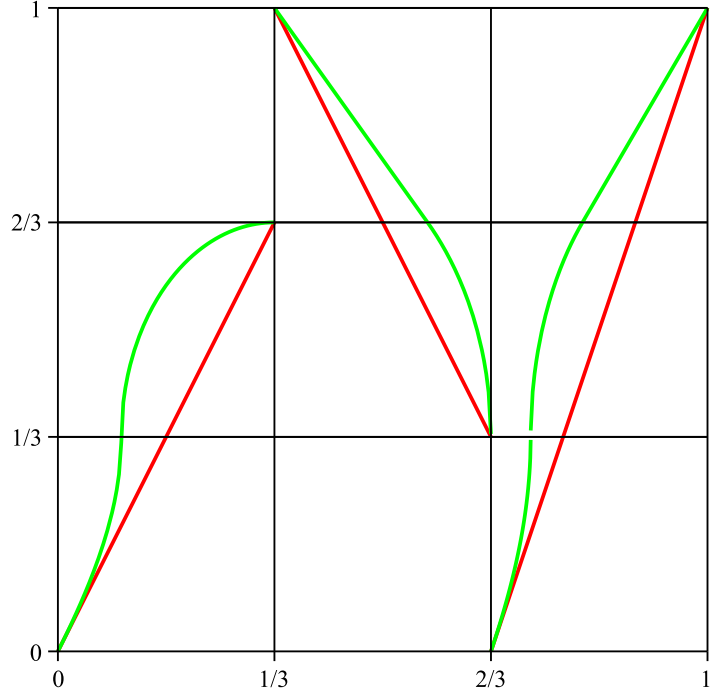}
\caption{Piecewise linear Markov map $\protect\tau_1$ and conjugated map $%
\protect\tau_2$}
\label{fig:Markov_example}
\end{figure}

Let us now consider a more general multi-valued map $T$ with lower edge
selection $\tau_1$ and upper edge selection $\tau_2$ as in Figure \ref%
{fig:Markov_example} or \ref{fig:counter_1}. Map $T$ is typically
infinitely valued $T(x)=[\tau_1(x),\tau_2(x)]$, $x\in[0,1]$. If $\tau_1$
preserves a density $f_1$ and $\tau_2$ preserves a density $f_2$, then we
ask whether for any convex combination $f=\alpha\cdot f_1+(1-\alpha)\cdot
f_2 $, $0<\alpha<1$, we can find a selection of $T$ which preserves the
density $f$. We give conditions under which this holds.

\begin{theorem}
\label{Th:linear} Let $\tau_1$ be a piecewise linear Markov (on partition $%
\mathcal{P}$) map of interval $[0,1]$ onto itself preserving density $%
f_1=\sum_{i=1}^n c_i\chi_{I_i}$. Let $\tau_2$ be a map conjugated to $\tau_1$
by an increasing absolutely continuous homeomorphism $h$ preserving
partition $\mathcal{P}$,
\begin{equation*}
\tau_2=h^{-1}\circ \tau_1\circ h\ .
\end{equation*}
Let density $f_2$ be $\tau_2$ invariant. Then, for any convex combination $%
f=\alpha\cdot f_1+(1-\alpha)\cdot f_2$, $0<\alpha<1$, we can find a
selection of $T$ which preserves the density $f$.
\end{theorem}

\begin{proof}
By Proposition \ref{Pr:conj} we have $|h^{\prime }|=f_2\cdot \sum_{i=1}^n
\frac 1{c_i}\chi_{I_i}\ ,$ so assuming that $h$ is increasing
\begin{equation*}
h(x)=\int_0^x f_2(t)\cdot \sum_{i=1}^n \frac 1{c_i}\chi_{I_i}(t) dt \ .
\end{equation*}
Using Proposition \ref{Pr:conj} again we see that if define conjugation
\begin{eqnarray*}
g(x)&=\int_0^x f(t)\cdot \sum_{i=1}^n \frac 1{c_i}\chi_{I_i}(t) dt= \int_0^x
\left(\alpha\cdot f_1(t)+(1-\alpha)\cdot f_2(t)\right)\cdot \sum_{i=1}^n
\frac 1{c_i}\chi_{I_i}(t) dt \\
&=\int_0^x \left(\alpha\cdot 1+(1-\alpha)\cdot \left(f_2(t)\cdot
\sum_{i=1}^n \frac 1{c_i}\chi_{I_i}(t)\right)\right) dt =\alpha\cdot x+
(1-\alpha)\cdot h(x)\ ,
\end{eqnarray*}
then $\tau=g^{-1}\circ \tau_1\circ g$ preserves density $f$. Note, that $g$
is also increasing.

We will prove that $\tau_1\le\tau\le\tau_2$. Consider $x\in I_i\in\mathcal{P}
$. Let $\beta=1-\alpha$. $\tau_1$ is piecewise linear on $I_i$, so $%
\tau_1(\alpha x+\beta y)=\alpha \tau_1(x)+\beta\tau_1(x)$, $x,y\in I_i$.
First, we will prove that $\tau_1\le \tau$ or equivalently $g\circ\tau_1\le
\tau_1\circ g$. For $x\in I_i$ we have
\begin{equation*}
g(\tau_1(x))=\alpha\tau_1(x)+\beta h(\tau_1(x))\le \alpha \tau_1(x)+\beta
\tau_1(h(x))=\tau_1(g(x))\ .
\end{equation*}
We used the inequality $h\circ\tau_1\le \tau_1\circ h$ equivalent to $%
\tau_1\le h^{-1}\circ\tau_1\circ h=\tau_2$.

Now, we prove that $\tau\le \tau_2$ or that $g^{-1}\circ \tau_1\circ g\le
\tau_2$ or equivalently that $\tau_1\circ g\le g\circ \tau_2$. Again, we
consider $x\in I_i$:
\begin{equation*}
\tau_1(g(x))=\tau_1(\alpha x+\beta h(x))=\alpha \tau_1(x)+\beta \tau_1(h(x))
\ .
\end{equation*}
We also have
\begin{equation*}
g (\tau_2(x))=\alpha \tau_2(x) +\beta h(\tau_2(x))=\alpha \tau_2(x) +\beta
h(h^{-1}(\tau_1( h(x))))=\alpha \tau_2 +\beta \tau_1( h(x))\ .
\end{equation*}
Since $\tau_1\le \tau_2$ the proof is completed.
\end{proof}

In this example we show existence of the "in between" map $\tau$ in a
situation when the lower map is not onto. Let us consider the tent map
\begin{equation*}
\tau_2(x)=1-2|x-1/2|\ ,
\end{equation*}
and
\begin{equation*}
\tau_1(x)=%
\begin{cases}
4x^2 & \ ,\ \text{for} \ 0\le x<1/4\ , \\
2x-1/4 & \ ,\ \text{for} \ 1/4\le x<1/2\ , \\
-2x+7/4 & \ ,\ \text{for} \ 1/2\le x<3/4\ , \\
4(1-x)^2 & \ ,\ \text{for} \ 3/4\le x\le 1\ ,%
\end{cases}%
\end{equation*}
shown in Figure \ref{fig:counter_1}.

\begin{figure}[h]
% float placement: (h)ere, page (t)op, page (b)ottom, other (p)age
\centering
% file name: E:/0Docs/0TeX/PapersT/Abrahams_MultiValued/counter_1.eps
\includegraphics[bb=20 118 575
673,width=3.92in,height=3.92in,keepaspectratio]{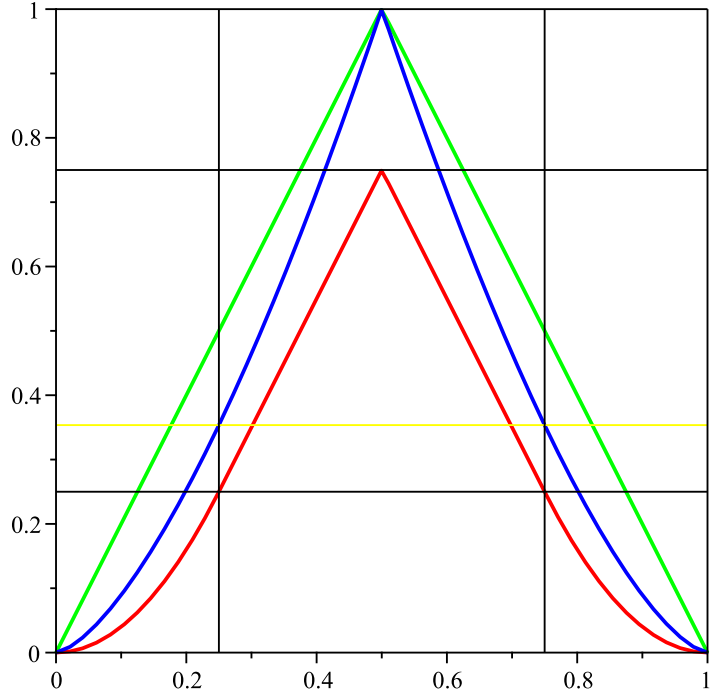}
\caption{Maps $\protect\tau_1$, $\protect\tau_2$ and map $\protect\tau$ we
are looking for.}
\label{fig:counter_1}
\end{figure}
The invariant densities are $f_2=1$ for $\tau_2$ and $f_1=2\chi_{[1/4,3/4]}$
for $\tau_2$. For any $0<\lambda<1$ their convex combination is
\begin{equation*}
f=\lambda f_1+(1-\lambda) f_2= (1-\lambda)\chi_{[0,1/4]\cup[3/4,1]%
}+(1+\lambda)\chi_{[1/4,3/4]}\ .
\end{equation*}
We are looking for map $\tau$ satisfying $\tau_1\le\tau\le\tau_2$ and
preserving $f$. We must have $\tau(0)=\tau(1)=0$ since both $\tau_1$ and $%
\tau_2$ satisfy these conditions. We also must have $\tau(1/2)=1$ since $f$
is supported on the whole $[0,1]$. We will look for a symmetric map $\tau$.

\begin{figure}[h]
% float placement: (h)ere, page (t)op, page (b)ottom, other (p)age
\centering
% file name: E:/0Docs/0TeX/PapersT/Abrahams_MultiValued/counter_2.eps
\includegraphics[bb=20 118 575
673,width=3.92in,height=3.92in,keepaspectratio]{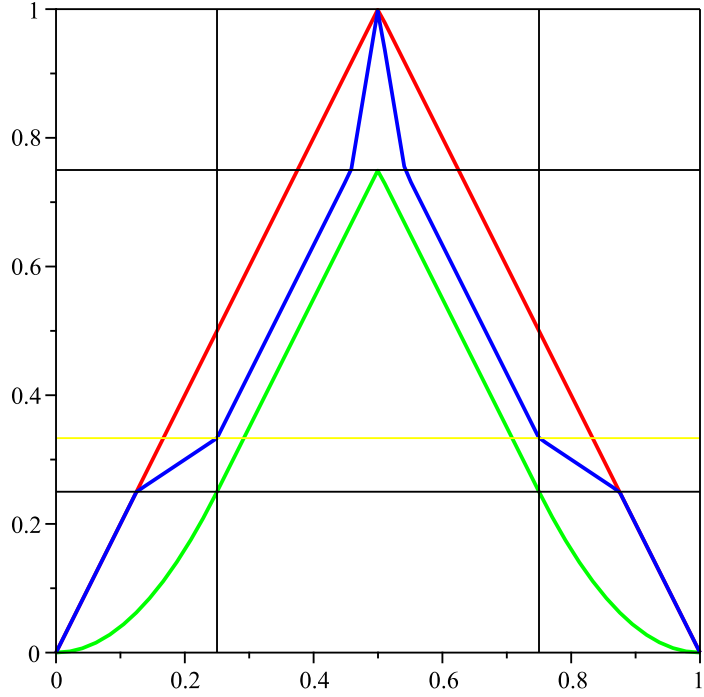}
\caption{Map $\protect\tau$ for $\protect\lambda=1/2$.}
\label{fig:counter_2}
\end{figure}
For $x\in[0,1/4]$ we have
\begin{equation*}
(1-\lambda)=\frac{(1-\lambda)}{\left|\tau^{\prime
}(\tau_{(1)}^{-1}(x))\right|} +\frac{(1-\lambda)}{\left|\tau^{\prime
}(\tau_{(2)}^{-1}(x))\right|}\ ,
\end{equation*}
or
\begin{equation*}
1=\frac{1}{\left|\tau^{\prime }(\tau_{(1)}^{-1}(x))\right|} +\frac{1}{%
\left|\tau^{\prime }(\tau_{(2)}^{-1}(x))\right|}\ .
\end{equation*}
Thus, by symmetry of $\tau$:
\begin{equation*}
|\tau^{\prime }(x)|=2 \ , \ \text{for}\ \ x\in [0,\tau_{(1)}^{-1}(1/4)]\cup
[1-\tau_{(1)}^{-1}(1/4),1] \ .
\end{equation*}
For $x\in[1/4,\tau(1/4)]$ we have
\begin{equation*}
(1+\lambda)= \frac{(1-\lambda)}{\left|\tau^{\prime
}(\tau_{(1)}^{-1}(x))\right|} +\frac{(1-\lambda)}{\left|\tau^{\prime
}(\tau_{(2)}^{-1}(x))\right|}\ ,
\end{equation*}
which, by symmetry of $\tau$, implies
\begin{equation*}
|\tau^{\prime }(x)|=2(1-\lambda)/(1+\lambda) \ , \ \text{for}\ \ x\in
[\tau_{(1)}^{-1}(1/4),1/4]\cup [3/4,1-\tau_{(1)}^{-1}(1/4)] \ .
\end{equation*}
For $x\in[\tau(1/4),3/4]$ we have
\begin{equation*}
(1+\lambda)= \frac{(1+\lambda)}{\left|\tau^{\prime
}(\tau_{(1)}^{-1}(x))\right|} +\frac{(1+\lambda)}{\left|\tau^{\prime
}(\tau_{(2)}^{-1}(x))\right|}\ ,
\end{equation*}
which, by symmetry of $\tau$, implies
\begin{equation*}
|\tau^{\prime }(x)|=2 \ , \ \text{for}\ \ x\in
[1/4,\tau_{(1)}^{-1}(3/4)]\cup [1-\tau_{(1)}^{-1}(3/4),3/4] \ .
\end{equation*}
For $x\in[3/4,1]$ we have
\begin{equation*}
(1-\lambda)= \frac{(1+\lambda)}{\left|\tau^{\prime
}(\tau_{(1)}^{-1}(x))\right|} +\frac{(1+\lambda)}{\left|\tau^{\prime
}(\tau_{(2)}^{-1}(x))\right|}\ ,
\end{equation*}
which, by symmetry of $\tau$, implies
\begin{equation*}
|\tau^{\prime }(x)|=2(1+\lambda)/(1-\lambda) \ , \ \text{for}\ \ x\in
[\tau_{(1)}^{-1}(3/4),1-\tau_{(1)}^{-1}(3/4)] \ .
\end{equation*}
In Figures \ref{fig:counter_2} and \ref{fig:counter_3} we present graphs of
map $\tau$ for $\lambda=1/2$ and $\lambda=1/10$, correspondingly. The slopes
are $2,2/3,2,6$ for the first and $2,18/11,2,22/9$ for the second.

\begin{figure}[h]
% float placement: (h)ere, page (t)op, page (b)ottom, other (p)age
\centering
% file name: E:/0Docs/0TeX/PapersT/Abrahams_MultiValued/counter_3.eps
\includegraphics[bb=20 118 575
673,width=3.92in,height=3.92in,keepaspectratio]{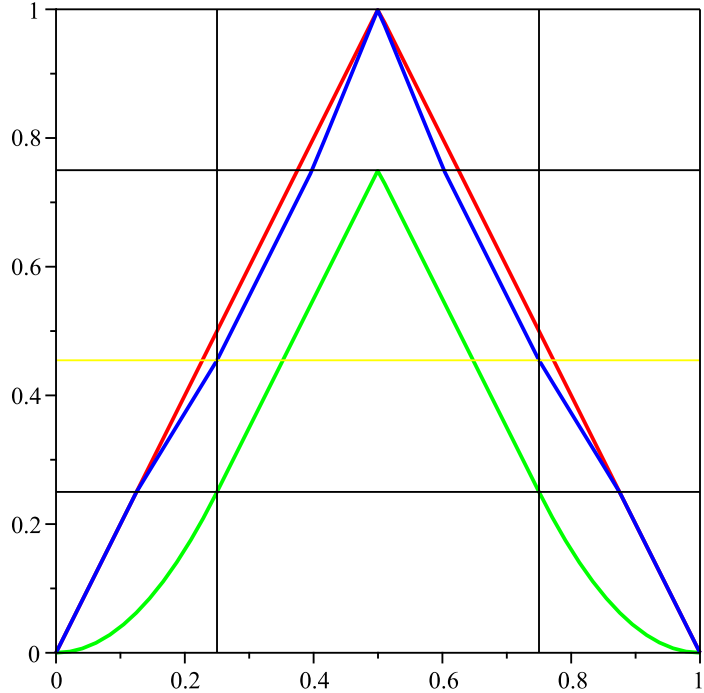}
\caption{Map $\protect\tau$ for $\protect\lambda=1/10$.}
\label{fig:counter_3}
\end{figure}

\section{ Multivalued Maps and Random Maps}

We define a random map to be a finite collection of maps as follows: let $%
T=(\tau _{1},\tau _{2},\dots ,\tau _{K};p_{1},p_{2},\dots ,p_{K})$, where $%
\tau _{k}$ are maps of an interval and $p_{k}$ are position dependent
probabilities, that is, $p_{k}(x)\geq 0$ for $k=1,2,....,K$ and $%
\sum_{k=1}^{K}p_{k}(x)=1$. At each step, the random map $T$ moves the point $%
x$ to $\tau _{k}(x)$ with probability $p_k(x)$. For fixed $\{\tau _{1},\tau
_{2},\dots ,\tau _{K}\}$, $T$ can have different invariant probability
density functions, depending on the choice of the (weighting)\ functions $%
\{p_{1},p_{2},\dots ,p_{K}\}$. Let $f_{k}$ be an invariant density of $\tau
_{k}$, $k=1,\dots ,K$. It is shown in \cite{BGR1,BGR2} that for any positive
constants $a_{k}$, $k=1,\dots ,K$, there exists a system of weighting
probability functions $p_{1},\dots ,p_{K}$ such that the density $%
f=a_{1}f_{1}+\dots +a_{K}f_{K}$ is invariant under the random map $T=\{\tau
_{1},\dots ,\tau _{K};p_{1},\dots ,p_{K}\},$ where
\begin{equation*}
p_{k}=\frac{a_{k}f_{k}}{a_{1}f_{1}+\dots +a_{K}f_{K}}\ \ ,\ \ k=1,2,\dots ,K,
\end{equation*}%
(It is assumed that $0/0=0$.)

Let us\ now consider a multivalued map consisting of a lower map $\tau _{1}$
and an upper map $\tau _{2},$ with density functions $f_{1}$ and $f_{2},$
respectively. Let $f$  be any convex combination of $f_{1}$ and $f_{2}.$
Then by the foregoing result we can construct a position dependent random
map on the graphs of $\tau _{1}$ and $\tau _{2}$ whose unique pdf is $f.$

A related problem  is to consider a piecewise expanding map $\tau
$ (a selection) between the graphs of $\tau _{1}$ and $\tau _{2}$ having density function
 $f.$ Can we find a probability function $p(x)$ such that the resulting random
map $T=(\tau _{1},\tau _{2};p,1-p)$ has $f$ as its density function?

In general this problem does not have a positive solution (see Example \ref{CEX} below).
In many cases the solution can be found.
A simple example of this situation  can be shown from Example \ref{counterEX}.
Let us consider the
triangle map, $\tau ,$ whose graph fits in between the graphs of $\tau _{1}$
and $\tau _{2}.$ It can be shown that $T=(\tau _{1},\tau _{2};0.75, 0.25)$ has
Lebesgue measure as its invariant measure.

Another, more general result in this direction, can be established by
considering $\tau _{1}$ to be a piecewise linear Markov map where $\tau _{2}$
is conjugated to $\tau _{1}$ by $g(x)=\alpha x+(1-\alpha )h(x),$ where $h$
conjugates the upper map $\tau _{2}$ to the lower map $\tau _{1}.$ Then $%
\tau $ has pdf $f$ which is a convex combination of $f_{1}$ and $f_{2}.$
Hence by the main result of \cite{BGR2}, we know there exists a position
dependent random map $T=(\tau _{1},\tau _{2};p,1-p)$ which has $f$ as its
pdf.

\begin{example}\label{CEX}
\end{example}

\begin{figure}[h] % float placement: (h)ere, page (t)op, page (b)ottom, other (p)age
  \centering
  % file name: F:/Abrahams_MultiValued/tau16.eps
  \includegraphics[bb=20 118 575 673,width=3.92in,height=3.92in,keepaspectratio]{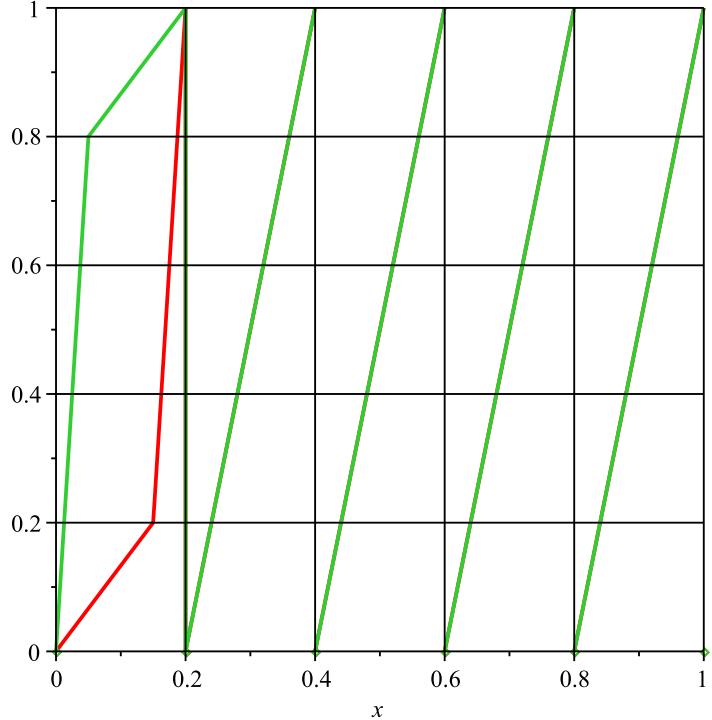}
  \caption{Maps in the counterexample}
  \label{fig:tau16}
\end{figure}

We consider the semi-Markov piecewise linear maps
$$\tau_1(x)=\begin{cases} \frac 43 x \ ,&\ \ \text{for}\ \ 0\le x<\frac 3{20}\ ;\\
                           16 x-\frac {11}5\ ,&\ \ \text{for}\ \ \frac 3{20}\le x<\frac 15\ ;\\
                           5x \ ({\rm mod})\ 1 \ ,&\ \ \text{for}\ \ \frac 15\le x\le 1\ ,\\
             \end{cases}
\ \ \ \ \ \tau_2(x)=\begin{cases} 16 x \ ,&\ \ \text{for}\ \ 0\le x<\frac 1{20}\ ;\\
                           \frac 43 x+\frac {11}{15}\ ,&\ \ \text{for}\ \ \frac 1{20}\le x<\frac 15\ ;\\
                           5x \ ({\rm mod})\ 1 \ ,&\ \ \text{for}\ \ \frac 15\le x\le 1\ .\\
             \end{cases}
$$

whose graphs are shown in Figure \ref{fig:tau16}. For the selection $\tau$ we choose the map $\tau(x)=5x$ (mod 1) preserving Lebesgue measure.

We will show that there is no solution, i.e., there is no position dependent random map based on
$\tau_1$, $\tau_2$ that  preserves  Lebesgue measure.

\begin{figure}[h] % float placement: (h)ere, page (t)op, page (b)ottom, other (p)age
  \centering
  % file name: F:/Abrahams_MultiValued/tau16_part.eps
  \includegraphics[bb=20 118 575 673,width=3.92in,height=3.92in,keepaspectratio]{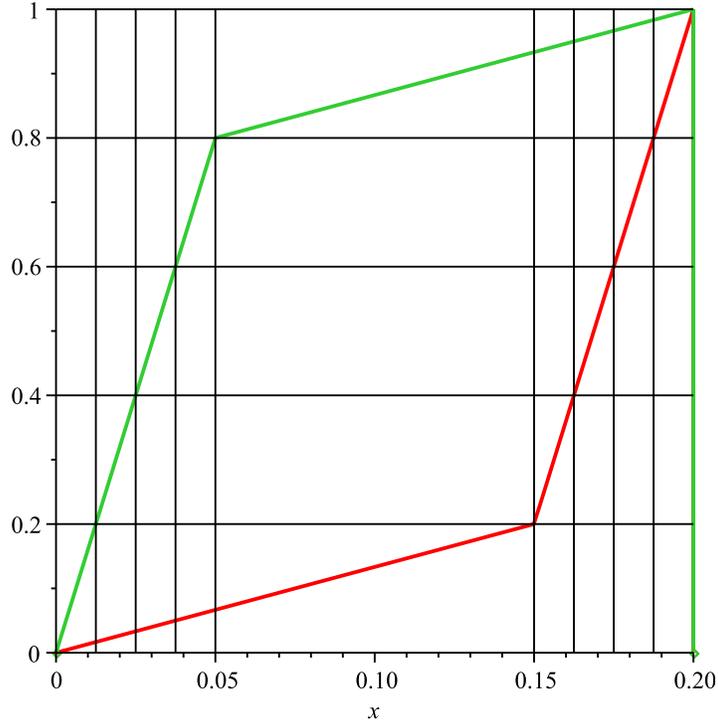}
  \caption{Maps $\tau_1$, $\tau_2$ on interval $[0,1/5]$ (not to scale).}
  \label{fig:tau16_part}
\end{figure}

 Let us consider $p_1(x)$ non-constant on $[0,1/5]$ (the values on $[1/5,1]$ are not important).
Let $\phi_1=\tau_1^{-1}$ on $[0,1/5]$, $\psi_1=\tau_2^{-1}$ on $[0,1/5]$, $\phi_i=\tau_1^{-1}=\tau_2^{-1}$ on $[(i-1)/5,i/5]$, $i=2,3,4,5$.
Frobenius-Perron operator of $\mathcal T$ is
\begin{equation}\label{equ}\begin{split}
P_{\mathcal T} f(x)&=\frac 34 p_1(\phi_1(x))f(\phi_1(x))\chi_{[0,1/5]}+
\frac 1{16}\left(1- p_1(\psi_1(x))\right)f(\psi_1(x))\chi_{[0,1/5]}\\
&+
\frac 1{16}p_1(\phi_1(x))f(\phi_1(x))\chi_{[1/5,4/5]}+
\frac 1{16}\left(1- p_1(\psi_1(x))\right)f(\psi_1(x))\chi_{[1/5,4/5]}\\
&+
\frac 1{16}p_1(\phi_1(x))f(\phi_1(x))\chi_{[4/5,1]}+
\frac 34\left(1- p_1(\psi_1(x))\right))f(\psi_1(x))\chi_{[4/5,1]}\\
&+
\frac 1 5 \left(f(\phi_2(x))+f(\phi_3(x))+f(\phi_4(x))+f(\phi_5(x))\right)\ .
\end{split}
\end{equation}
If we assume that $f=1$ is preserved by  $P_{\mathcal T}$, then equation (\ref{equ}) reduces to
\begin{equation}\label{equ2}\begin{split}
\frac 15&=\frac 34 p_1(\phi_1(x))\chi_{[0,1/5]}+
\frac 1{16}\left(1- p_1(\psi_1(x))\right)\chi_{[0,1/5]}\\
&+
\frac 1{16}p_1(\phi_1(x))\chi_{[1/5,4/5]}+
\frac 1{16}\left(1- p_1(\psi_1(x))\right)\chi_{[1/5,4/5]}\\
&+
\frac 1{16}p_1(\phi_1(x))\chi_{[4/5,1]}+
\frac 34\left(1- p_1(\psi_1(x))\right))\chi_{[4/5,1]}
\ .
\end{split}
\end{equation}

\begin{figure}[h] % float placement: (h)ere, page (t)op, page (b)ottom, other (p)age
  \centering
  % file name: E:/0Docs/0TeX/PapersT/Abrahams_MultiValued/tau21_map.eps
  \includegraphics[bb=20 118 575 673,width=3.92in,height=3.92in,keepaspectratio]{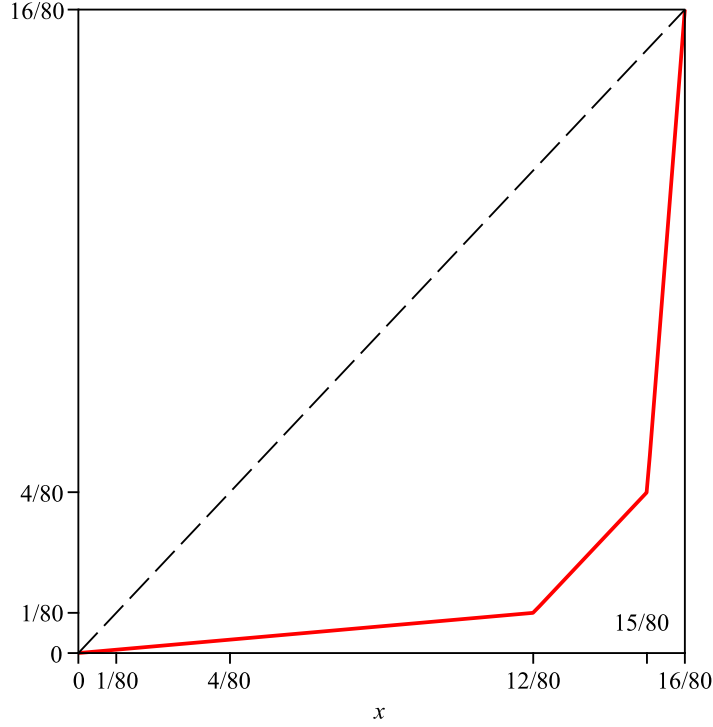}
  \caption{Map $\tau_{21}$ on $[0,1/5]$.}
  \label{fig:tau21_map}
\end{figure}

We introduce a map $\tau_{21}:[0,1/5]\to[0,1/5]$, $\tau_{21}=\tau_2^{-1}\circ\tau_1$, (see Fig. \ref{fig:tau21_map}) defined by
$$\tau_{21}(y)=\begin{cases}\frac 1{12} x \ ,&\ \ \text{for}\ \ 0\le x<\frac 3{20}\ ;\\
                            x-\frac {11}{80}\ ,&\ \ \text{for}\ \ \frac 3{20}\le x<\frac {15}{80}\ ;\\
                           12x -\frac {11}5 \ ,&\ \ \text{for}\ \ \frac {15}{80}\le x\le \frac{1}5\ .\\
             \end{cases}
$$
We assume that the solution $p_1$ exists and is a probability, i.e., its values are between 0 and 1. In particular it is defined on the interval $[1/80,4/80]$ and on interval $[12/80,15/80]$.
Let us  consider equation (\ref{equ2}) for $x\in [1/5,4/5]$.   We have

\begin{equation}\label{equ4}
\frac 15=
\frac 1{16}p_1(\phi_1(x))+
\frac 1{16}\left(1- p_1(\psi_1(x))\right)
\ ,
\end{equation}
or, substituting $x=\phi_1^{-1}(y)$, $y\in [12/80,15/80]$,
\begin{equation}\label{equ4}
\frac 15=
\frac 1{16}p_1(y)+
\frac 1{16}\left(1- p_1(\psi_1(\phi_1^{-1}(y)))\right)
\ .
\end{equation}
Using the equality $\tau_{21}=\psi_1\circ\phi_1^{-1}$, this can be rewritten as
$$ p_1(y)=\frac {11}5 +p_1(\tau_{21}(y))\ .$$
Note that $\tau_{21}([12/80,15/80])=[1/80,4/80]$. Whatever are the values of $p_1$ on $[1/80,4/80]$,
this implies that the values on $[12/80,15/80]$ are strictly larger than $1$.
This contradicts the assumptions on $p_1$.

\end{document}